\documentclass{amsart}
\usepackage{amsfonts}
\usepackage{graphicx}
\usepackage{amscd}
\usepackage{amsmath}
\usepackage{amssymb}
\usepackage{xcolor}

\makeatletter
\@namedef{subjclassname@2010}{%
  \textup{2010} Mathematics Subject Classification}
\makeatother

\usepackage{mathrsfs}
\usepackage{amsbsy}

\newtheorem{corollary}{\bf Corollary}
\newtheorem{definition}{\bf Definition}
\newtheorem{lemma}{\bf Lemma}

\newtheorem{proposition}{\bf Proposition}
\newtheorem{remark}{Remark}

\newtheorem{theorem}{\bf Theorem}

\theoremstyle{definition}
\newtheorem{example}{\bf Example}

\numberwithin{equation}{section}

\makeatletter
\@namedef{subjclassname@2020}{%
  \textup{2020} Mathematics Subject Classification}
\makeatother

\makeatletter
\@namedef{subjclassname@2020}{%
  \textup{2020} Mathematics Subject Classification}
\makeatother

\begin{document}
	
	\title[Static perfect fluid space-time]{Geometry of static perfect fluid space-time}
	\author{J. Costa, R. Di\'ogenes, N. Pinheiro}
	\author{E. Ribeiro Jr}

\address[J. Costa]{Universidade Federal do Cear\'a, Departamento de Matem\'atica. Campus do Pici, 60355-636 - Fortaleza, CE - Brazil.}\email{johnatansc@alu.ufc.br }	
	
\address[R. Di\'ogenes]{UNILAB, Instituto de Ci\^encias Exatas e da Natureza, Rua Jos\'e Franco de Oliveira, 62790-970, Reden\c{c}\~ao - CE, Brazil.}\email{rafaeldiogenes@unilab.edu.br}
	
\address[N. Pinheiro]{Universidade Federal do Amazonas, Instituto de Ci\^encias Exatas - ICE, Departamento de Matem\'atica. Campus Coroado, 69080900 - Manaus / AM, Brazil.}\email{neilha@ufam.edu.br}

\address[E. Ribeiro Jr]{Departamento  de Matem\'atica, Universidade Federal do Cear\'a - UFC, Campus do Pici, 60455-760, Fortaleza - CE, Brazil.}
\email{ernani@mat.ufc.br}

\thanks{J. Costa was partially supported by CAPES/Brazil - Finance Code 001.}	
\thanks{R. Di\'ogenes was partially supported by CNPq/Brazil [Grant: 310680/2021-2]}
	\thanks{E. Ribeiro was partially supported by CNPq/Brazil [Grant: 309663/2021-0 \& 403344/2021-2], CAPES/Brazil and FUNCAP/Brazil [Grant: PS1-0186-00258.01.00/21].}

\thanks{Corresponding Author: E. Ribeiro Jr (ernani@mat.ufc.br)}

\begin{abstract}
In this article, we investigate the geometry of static perfect fluid space-time on compact manifolds with boundary. We use the generalized Reilly’s formula to establish a geometric inequality for a static perfect fluid space-time involving the area of the boundary and its volume. Moreover, we obtain new boundary estimates for static perfect fluid space-time. One of the boundary estimates is obtained in terms of the Brown-York mass and ano\-ther one related to the first eigenvalue of the Jacobi operator. In addition, we provide a new (simply connected) counterexample to the Cosmic no-hair conjecture for arbitrary dimension $n\geq 4.$ 
\end{abstract}
	
\date{\today}
	
	
\maketitle
	
\section{Introduction}\label{SecInt}

The static space-times are special solutions to Einstein equations in general re\-la\-tivity and they play a fundamental role in modern physics and geometry. To be precise, given a Riemannian manifold $(M^n,\,g),$ $n\ge 3,$ and a positive smooth function $f$ on $M^n,$ we say that $(\widehat{M}^{n+1},\widehat{g})=M^n\times_f \mathbb{R}$ endowed with the metric $\widehat{g}=g-f^2 dt^2$ is a static space-time. In this situation, the Einstein equation with perfect fluid as a matter field over $(\widehat{M}^{n+1},\widehat{g})=M^n\times_f \mathbb{R}$ is given by
\begin{equation}\label{eqi1}
    Ric_{\widehat{g}}-\frac{R_{\widehat{g}}}{2}\widehat{g}=T,
\end{equation} where $T=\mu f^2 dt^2+\rho g$ is the stress-energy-momentum tensor of a perfect fluid, $Ric_{\widehat{g}}$ and $R_{\widehat{g}}$ stand for the Ricci tensor and the scalar curvature for the metric $\widehat{g},$ respectively. Furthermore, the smooth functions $\mu$ and $\rho$ are \textit{mass-energy density} and \textit{pressure} of the fluid (as measured in the rest frame). The fluid is called ``perfect" because of the absence of heat conduction terms and stress terms corresponding to viscosity. For more details, see, e.g., \cite{CLR,SPFS,Kobayashi,Kobayashi-Obata,Oliynyk,Wald}.

The perfect fluid space-times are natural generalizations of the static vacuum spaces and certain solutions of \eqref{eqi1} provide models for galaxies, black holes and stars (see \cite{Hawking, Oliynyk, Wald}). In particular, they are used in developing realistic stellar models (or models for fluid planets) and represent a homogeneous fluid filled universe that is undergoing accelerated expansion. Astronomical evidences also indicate that the universe can be modeled as a space-time containing a perfect fluid; see \cite{CLR,SPFS,Martin,Wald} and the references therein. 

In order to make our approach more comprehensible, we fix some terminology (cf. \cite{SPFS,Kobayashi-Obata,Shen-FM}). 

\begin{definition}\label{def-perfect-fluid}
A Riemannian manifold $\big(M^n,\,g\big)$ is said to be a spatial factor of a static perfect fluid space-time if there exist smooth functions $f>0$ and $\rho$ on $M^n$ satisfying

\begin{eqnarray}\label{eq-general}
f\mathring{Ric} = \mathring{\nabla}^2 f
\end{eqnarray}
and
\begin{eqnarray}\label{eq-laplacian}
\Delta f=\left(\frac{n-2}{2(n-1)}R+\frac{n}{n-1}\rho\right)f,
\end{eqnarray} where $\mathring{Ric}$, $\mathring{\nabla}^2 f$ and $R$ stand for the traceless Ricci, traceless Hessian of $f$ and scalar curvature of $(M^n,\,g),$ respectively. When $M^n$ has non-empty boundary $\partial M,$ we assume in addition that $f^{-1}(0)=\partial M.$ The function $f$ is usually called lapse function or static potential in the literature. In this case, $(M^n,\,g,\,f,\,\rho)$ will be called {\it static perfect fluid space-time} (SPFST). 
\end{definition}

One should be emphasized that the \textit{dominant energy condition} is said to be satisfied when $\mu\geq|\rho|,$ which means that the speed of the energy flow can not be equal or greater than to the light. As was observed by Hawking and Ellis \cite{Hawking}, the dominant energy condition holds for all known matter; see also \cite[ p. 347]{O'Neil} and \cite[p. 219]{Wald}. In geometrical point of view, Eq. \eqref{eq-general} is related to important special metrics, as for instance, vacuum static spaces (see \cite{Ambrozio,BS,SPFS,Shen-FM}), Miao-Tam critical metrics or $V$-static spaces (see \cite{Baltazar-Diogenes-Ribeiro,BDR,Batista-Diogenes-Ranieri-Ribeiro,Corvino-Eichmair-Miao,Miao-Tam,Yuan2}) and Einstein-type manifolds (see \cite{Besse,CSW,hpw1,hpw2,kk,compact}). To be precise, $V$-static spaces are critical metrics of the volume functional on a given manifold $M^n$ when restricted to the class of metrics with prescribed constant scalar curvature for a prescribed Riemannian metric on the boundary $\partial M.$ One easily verifies from \eqref{eq-general} and \eqref{eq-laplacian} that $\mu=\frac{R}{2}.$ Moreover, as pointed out by Coutinho \textit{et al.} \cite[Proposition 2]{SPFS}, in contrast with static spaces and $V$-static spaces, the static perfect fluid space-time equation alone does not imply that the scalar curvature is constant (see \cite{BLP,Masood}).

A classical example of SPFST with connected (non-empty) boundary is the $n$-dimensional hemisphere $\mathbb{S}^{n}_{+}(r)$ of radius $r$ endowed with the standard metric $g_{can}$ and potential function $f(h)=\cos(h)$, where $h\leq\frac{\pi}{2}$ is the height function. In this case, $\partial M=\mathbb{S}^{n-1}(r).$ For the case of disconnected boundary, we have $[0,\pi]\times\mathbb{S}^{n-1}.$ To be precise, let $M=[0,\pi]\times\mathbb{S}^{n-1}$ be a Riemannian product with metric $g=dt^2+(n-2)g_{_{\mathbb{S}^{n-1}}}$ and potential function $f(t)=\sin(t).$ Thus, $M$ is a compact oriented static perfect fluid space-time with disconnected boundary (the boundary is the union of two copies of $\mathbb{S}^{n-1}$). For more examples, see, e.g., \cite{BLP,Masood}.

It has been conjectured in 1984 by Boucher, Gibbons and Horowitz in \cite{BGH,BG} that: {\it the hemisphere $\mathbb{S}^{n}_{+}$ is the only possible $n$-dimensional (simply connected) positive static triple with single-horizon (connected boundary) and positive scalar curvature.} This conjecture is known as {\it Cosmic no-hair conjecture}. It is closely related to Fischer-Marsden conjecture which asserts that the standard unit round spheres $(\mathbb{S}^n,\,g_{can})$ are the only closed Riemannian manifolds with scalar curvature $n(n-1)$ admitting static potential (see \cite{Ambrozio,Shen-FM}).  
In the last decades some partial answers to the Cosmic no-hair conjecture were obtained. For instance, by assuming that $(M^{n},\,g)$ is Einstein, it suffices to apply the Obata type theorem due to Reilly \cite{Reilly1977} to conclude that the conjecture is true. Moreover, Kobayashi \cite{Kobayashi} and Lafontaine \cite{lafontaine} proved independently that such a conjecture is also true under conformally flat condition. Qing and Yuan \cite{jiewei} proved the Cosmic no-hair conjecture by considering a weaker hypothesis on the Cotton tensor. In the work \cite{GHP}, Gibbons, Hartnoll and Pope constructed counterexamples to the Cosmic no-hair conjecture in the cases of dimension $4\leq n\leq 8,$ but their counterexamples are not simply connected.

In the sequel, we provide a simply connected counterexample to the Cosmic no-hair conjecture for arbitrary dimension $n\ge 4.$

\begin{example}
\label{exA}
Let $M^n=\mathbb{S}^{p+1}_{+}\times\mathbb{S}^q$, $q>1$, with the doubly warped product metric 
	\begin{eqnarray*}
		g=dr^2+\sin^2(r)g_{\mathbb{S}^p}+\frac{q-1}{p+1}g_{\mathbb{S}^q},
	\end{eqnarray*} where $r(x,y)=r(x)$ is the height function of $\mathbb{S}^{p+1}_{+}$. By considering the potential function $f(r)=\cos(r)$ and  $r\leq\frac{\pi}{2},$ one obtains that $(M^n,\,g)$ must satisfy the Eqs. \eqref{eq-general} and \eqref{eq-laplacian}. 
		\end{example}

\begin{remark}
The above example has positive constant scalar curvature and hence, it is clearly a positive static triple. Besides, it is simply connected. Example \ref{exA} is inspired in a new example of $m$-quasi-Einstein manifold with boundary obtained by Ribeiro et al. in \cite{DGRPAMS}. A detailed description of Example \ref{exA} will be presented in Section \ref{Sec2}. 
\end{remark}

Geometric inequalities are classical objects of study in geometry and physics. In the recent years, boundary and volume estimates, as for example, isoperimetric and Shen-Boucher-Gibbons-Horowitz type inequalities, have been attracted a large interest because they are useful in proving new classification results and put away some possible examples of special metrics on a given manifold. Recently, the Reilly's formula \cite{Reilly1977} have been shown a promising tool to gain new geometric inequalities and obstruction results. In  \cite{Ros}, Ros used the Reilly's formula to prove the Alexandrov’s rigidity theorem for high order mean curvatures. Besides, Miao, Tam and Xie  \cite{Miao-Tam-N.-Q} used the Reilly's formula to obtain a stability inequality for Wang-Yau energy. A similar result was obtained by Kwong and Miao \cite{KM} to the boundary of static spaces. In the work \cite{Qiu-Xia}, Qiu and Xia proved a generalized Reilly’s formula that was used to give an alternative proof of the Alexandrov’s theorem and prove a new Heintze-Karcher inequality for Riemannian manifolds with boundary and sectional curvature bounded from below. Subsequently, Xia \cite{Xia} used the generalized Reilly’s formula to establish a Minkowski type inequality for weighted mixed volumes in non-Euclidean space forms. Very recently, Di\'ogenes, Pinheiro and Ribeiro \cite{DPR} used the generalized Reilly’s formula by Qiu and Xia to obtain sharp integral estimates for critical metrics of the volume functional that were used to obtain a sharp boundary estimate for such metrics. 

In our first result, we shall use the generalized Reilly’s formula by Qiu and Xia (Proposition \ref{prop-qui-xia}) to establish a new boundary estimate for SPFST. More precisely, we have the following result.

\begin{theorem}
\label{th1B}
Let $\big(M^{n},\,g,\,f,\,\rho\big)$ be a compact oriented static perfect fluid space-time with connected boundary $\partial M$ and positive scalar curvature satisfying
\begin{eqnarray}\label{cte}
    \frac{n-2}{2(n-1)}R+\frac{n}{n-1}\rho=-\Lambda,
\end{eqnarray} where $\Lambda$ is a positive constant. Then we have:
\begin{eqnarray}
\label{th1-est01}
Vol(M)\geq\frac{1}{\Lambda}\sqrt{\frac{R_{min}+3\Lambda}{2n}}\,|\partial M|,
\end{eqnarray} where $R_{min}$ is the minimum value of $R$ on $M^n$ and $|\partial M|$ is the area of $\partial M.$ Moreover, if equality holds in (\ref{th1-est01}), then $(M^n,\,g)$ is isometric to the round hemisphere $\mathbb{S}^{n}_{+}.$
\end{theorem}

Notice that Theorem \ref{th1B} can be seen as an obstruction result, in terms of the area of the single-horizon (connected boundary), for the existence of a static perfect fluid space-time with positive scalar curvature on a given compact manifold $M^n$ with boundary $\partial M.$

\begin{remark}\label{remarkLambda}
As was mentioned, the constant $\Lambda$ in Theorem \ref{th1B} is positive. Indeed, supposing that $\Lambda\leq0,$ since $\Delta f=-\Lambda f$ and $f$ is a non-negative function with $f^{-1}(0)=\partial M,$ we may use the Maximum Principle to infer that $f=0$ in $M,$ which leads to a contradiction.
\end{remark}

\begin{remark}
We highlight that by assuming the dominant energy condition in Theorem \ref{th1B}, one obtains that the scalar curvature of $M^n$ must be positive. In fact, the dominant energy condition asserts $\frac{R}{2}\geq|\rho|$ and hence, if $R(p)=0$ for some point $p\in M,$ then $\rho(p)=0,$ which contradicts the assumption that $\Lambda$ is a positive constant. Moreover, we have from Proposition 2 of \cite{SPFS} that $\rho_{|_{\partial M}}=-\frac{1}{2}R_{|_{\partial M}}$ and consequently, $\Lambda=\frac{1}{n-1}R_{|_{\partial M}}.$ In particular, \eqref{cte} implies that the scalar curvature is constant along the boundary.
\end{remark}

As a consequence of Theorem \ref{th1B}, we obtain the following corollary.

\begin{corollary}\label{cor1}
Let $\big(M^{n},\,g,\,f,\,\rho\big)$ be a compact oriented static perfect fluid space-time with connected boundary $\partial M$ and constant positive scalar curvature $R.$ Then we have:
\begin{eqnarray}\label{cor1-est01}
Vol(M)\geq\sqrt{\frac{(n-1)(n+2)}{2nR}}\,|\partial M|.
\end{eqnarray} Moreover, if equality holds in (\ref{cor1-est01}), then $(M^n,\,g)$ is isometric to the round hemisphere $\mathbb{S}^{n}_{+}$.
\end{corollary}

Before proceeding, it is fundamental to recall the definition of Riemannian Brown-York mass. Let $\Sigma$ be a connected hypersurface in $(M^n,\,g)$ such that $(\Sigma, g_{|_{\Sigma}})$ can be embedded in $\mathbb{R}^n$ as a convex hypersurface. Then, the Riemannian Brown-York mass $\mathfrak{m}_{_{BY}}$ of $\Sigma$ with respect to $g$ is given by 
\begin{eqnarray*}
    \mathfrak{m}_{_{BY}}(\Sigma,g)=\int_{\Sigma}(H_0-H_g) dS_g,
\end{eqnarray*} where $H_0$ and $H_g$ are the mean curvature of $\Sigma$ as hypersurface of $\mathbb{R}^n$ and $M,$ respectively, and $dS_g$ is the volume element of  on $\Sigma$ induced by $g.$ In \cite{Yuan}, motivated by the Riemannian Penrose inequality, Yuan obtained a boundary estimate for static spaces in terms of the Riemannian Brown-York mass. A similar result was established for quasi-Einstein manifolds by Di\'ogenes, Gadelha and Ribeiro \cite{DGR}. In another direction, inspired by ideas outlined in \cite{SPFS}, Andrade and Melo \cite{Andrade-Melo} proved recently that, under suitable conditions, the Hawking mass of Einstein-type manifolds is bounded from below by the area of the boundary.

In our next result, we shall establish a sharp boundary estimate for compact static perfect fluid space-time with (possibly disconnected) boundary in terms of the Riemannian Brown-York mass $\mathfrak{m}_{_{BY}},$ which can be also seen as an obstruction result. Our next result is the following.

\begin{theorem}
\label{theo2}
Let $\big(M^n,\,g,\,f,\,\rho\big)$, $n\geq 3$, be a compact static perfect fluid space-time with (possibly disconnected) boundary and positive scalar curvature satisfying the dominant energy condition. Suppose that each boundary component $(\partial M_i,\,g)$ can be isometrically embedded in $\mathbb{R}^n$ as a convex hypersurface. Then we have
\begin{eqnarray}\label{eqq4}
    |\partial M_i|\leq c\, \mathfrak{m}_{_{BY}}(\partial M_i,g),
\end{eqnarray} where $c$ is a positive constant. Moreover, equality occurs for some component $\partial M_i$ if and only if $M^n$ is isometric to the standard hemisphere $\mathbb{S}^{n}_{+}$. 
\end{theorem}

A key ingredient in the proof of Theorem \ref{theo2} is the Riemannian positive mass theorem for Brown-York mass by Shi-Tam \cite{Shi-Tam}, which is equivalent to the (higher dimensional) positive mass theorem for ADM mass by Schoen and Yau \cite{SY1,SY2,SY3} and Lohkamp \cite{Lohkamp}. It should be mentioned that the isometric embedding condition in Theorem \ref{theo2} was needed to use the Riemannian positive mass theorem. According to the solution of the Weyl problem, the isometrical embedding assumption can be replaced by control on sectional curvatures, as for instance, positive Gaussian curvature when $n=3$, see, e.g., \cite{Fang-Yuan,Yuan}.

As an application of Theorem \ref{theo2} we have the following result.

\begin{corollary}\label{coro-BY}
Let $\big(M^n,\,g,\,f,\,\rho\big),$ $n\geq 3,$ be a compact static perfect fluid space-time with (possibly disconnected) boundary and positive scalar curvature. Assume  the dominant energy condition and that each boundary component $(\partial M_i,g)$ can be isometrically embedded in $\mathbb{R}^n$ as a convex hypersurface. Then we have
\begin{eqnarray*}
|\partial M_i|\leq \widetilde{c}\int_{\partial M_i}(R^{\partial M_i}+|\mathring{h_i}|^2) dS_g
\end{eqnarray*} for some positive constant $\widetilde{c},$ where $\mathring{h_i}$ is the traceless second fundamental form of $\partial M_i$ as a hypersurface of $\mathbb{R}^n$ and $R^{\partial M_i}$ is the scalar curvature of $\partial M_i.$ Moreover, equality occurs for some connected component of the boundary if and only if $(M^n,\,g)$ is isometric to the round hemisphere $\mathbb{S}^{n}_{+}.$ 
\end{corollary}

As was already mentioned, geometric inequalities are useful to obtain new obstruction results. In this context, since the boundary $\partial M$ of a static perfect fluid space-time is compact and totally geodesic (with the induced metric, see Section \ref{Sec2}), it is natural to seek for an estimate for the area of the boundary $\partial M$ in terms of the eigenvalue of Jacobi operator of $\partial M.$ Before to state our next result, we recall that given an arbitrary function $\varphi\in C^{\infty}(\partial M),$ we define the Jacobi operator (or stability operator) $J$ acting in $\varphi$ as
\begin{eqnarray*}
    J({\varphi})=\Delta_{_{\partial M}}\varphi +(R_{nn}+|h|^2)\varphi,
\end{eqnarray*} where $\Delta_{_{\partial M}}$ stands for the Laplacian operator on $\partial M$, $R_{nn}$ is the Ricci curvature of $M$ in the direction of the outward unit normal vector field $\nu=-\frac{\nabla f}{|\nabla f|}$ and $h$ is the second fundamental form of $\partial M;$ for more details, see Section \ref{Sec2}. Furthermore, let $\lambda_1$ be the first eigenvalue of the Jacobi operator $J,$ i.e., $J(\varphi)=-\lambda_1 \varphi.$ Thus, $\lambda_1$ is given by
\begin{eqnarray}\label{eq-egvle}
    \lambda_1=\inf_{\varphi\neq 0} \frac{-\int_{\partial M}\varphi J(\varphi) dS_g}{\int_{\partial M}\varphi^2 dS_g}.
\end{eqnarray} The \textit{index} of $\partial M$ is the number of (counted with multiplicity) negative eigenvalues of $J.$ Moreover, when the index is zero, we say that the boundary $\partial M$ is \textit{stable}. Recently, assuming that the boundary is Einstein and has positive scalar curvature, Barros and Silva \cite[Theorem 1.10]{BS} proved an estimate for the area of the boundary of a triple static space depending on the first eigenvalue of the Jacobi operator. Here, inspired by \cite{BS} and ideas outlined in \cite[Theorem 1]{SPFS}, we shall establish a similar estimate for the area of the boundary $\partial M$ of a static perfect fluid space-time. More precisely, we have the following result.

\begin{theorem}\label{theo2.1}
Let $(M^n,g,f,\rho)$ be a compact oriented static perfect fluid space-time with connected Einstein boundary and $R^{\partial M}>0.$ Then, one has
\begin{equation}\label{eq-th2.1}
\left(R_{*}+2\lambda_{1}\right)|\partial M|^{\frac{2}{n-1}}\leq (n-1)(n-2)(\omega_{n-1})^{\frac{2}{n-1}},
\end{equation} where $R_{*}=\min\{R(p);\,p\in\partial M\}$ and $\omega_{n-1}$ denotes the volume of the round unit sphere $\Bbb{S}^{n-1}.$ Moreover, if equality occurs in \eqref{eq-th2.1}, then $(\partial M,g_{_{\partial M}})$ is isometric, up to scaling, to the standard sphere $\mathbb{S}^{n-1}.$
\end{theorem}

	\medskip

This article is organized as follows. In Section \ref{Sec2}, we review some basic facts and useful lemmas on static perfect fluid space-time (SPFST) that will be used in the proofs of the main results. Example \ref{exA} is also discussed in Section \ref{Sec2}. Finally, Section 3 collects the proofs of Theorems \ref{th1B}, \ref{theo2} and \ref{theo2.1} and Corollaries \ref{cor1} and \ref{coro-BY}.

\medskip

{\bf Acknowledgement.} The authors want to thank the anonymous referees for their careful reading, relevant remarks and valuable suggestions. Moreover, the authors want to thank R. Batista and B. Leandro for their interest in this work and for pointing out the references \cite{BLP,Masood}.

\section{Preliminaries}
\label{Sec2}

In this section, we recall some basic facts and useful lemmas. Initially, we remember that a static perfect fluid space-time (SPFST) is a Riemannian manifold $(M^n,\,g)$ jointly with a positive smooth function $f$ on $M^n$ satisfying 

\begin{eqnarray}
\label{eqA}
f\mathring{Ric} = \mathring{\nabla}^2 f
\end{eqnarray}
and
\begin{eqnarray}
\label{eqB}
\Delta f=\left(\frac{n-2}{2(n-1)}R+\frac{n}{n-1}\rho\right)f.
\end{eqnarray} We also recall that $f^{-1}(0)=\partial M,$ which implies that $\Delta f=0$ along $\partial M.$ In particular, it is easy to check that SPFST with constant scalar curvature reduces to a static triple space.

It follows from \cite[Lemma 1]{Leandro-Pina-Ribeiro} that $|\nabla f|$ does not vanish on the boundary. We present here an alternative proof of this fact.

\begin{proposition}\label{prop2}
Let $(M^n,\,g,\,f,\,\rho)$ be a compact static perfect fluid space-time with boundary $\partial M.$ Then $|\nabla f|$ is a nonzero constant along $\partial M.$
\end{proposition}

\begin{proof} Since $f$ vanishes on $\partial M,$ one sees from (\ref{eqA}) and (\ref{eqB}) that
\begin{eqnarray*}
X(|\nabla f|^2)=2\nabla^2 f(\nabla f,X)=0,
\end{eqnarray*} for any $X\in \mathfrak{X}(\partial M).$ Hence, $|\nabla f|$ is constant along $\partial M.$ Now, we need to show that $|\nabla f|_{|_{\partial M}}\neq 0.$ Indeed, let $p$ be a point in $\partial M$ and $\gamma:[0,\varepsilon)\rightarrow M$ be a geodesic parametrized by arc length with $\gamma(0)=p$ and $\gamma'(0)\perp \partial M.$ Choosing $u(t)=(f\circ\gamma)(t)$, one has 
\begin{eqnarray*}
	u''(t)=\nabla^2 f(\gamma'(t),\gamma'(t)).
\end{eqnarray*}
Then, it follows from \eqref{eqA} and \eqref{eqB} that there exists a smooth function $F(t)$ so that  $$u''(t)=F(t)u(t).$$ Consequently, 
\begin{eqnarray*}
	\left\{\begin{array}{ccl}
		u''(t)&=&F(t)u(t),\\
		u'(0)&=&g(\nabla f(p),\gamma'(0)),\\
		u(0)&=&f(p)=0.
	\end{array}\right.	
\end{eqnarray*} So, by assuming that $\nabla f(p)=0,$ one deduces that  $u'(0)=0$ and then, by using the existence and uniqueness theorem for ODE, we infer that $u=0$ over a neighborhood of $t=0$ in $[0,\varepsilon),$ which leads to a contradiction with the fact that $f>0$ in the interior of $M$ and $\gamma(t)$ lies in the interior of $M$ for all $t\in(0,\varepsilon_0)$ with sufficient small $\varepsilon_0>0.$ This concludes the proof of the proposition. 
\end{proof}

Proceeding, since $f$ is non-negative, one sees that $\nu=-\frac{\nabla f}{|\nabla f|}$ is the unit outward normal vector field of $\partial M.$ In particular, the divergence theorem implies that the integral of $\Delta f$ is not identically zero. In fact, observe that
\begin{eqnarray*}
\int_M \Delta f dV_g &=&\int_{\partial M}g(\nabla f,\nu) dS_g\nonumber\\&=&-|\nabla f|_{|_{\partial M}}|\partial M|\neq 0.
\end{eqnarray*}  This together with the dominant energy condition $\frac{R}{2}\geq |\rho|$ also implies that $R\geq 0,$ but not identically zero.

From now on, consider an orthonormal frame $\{e_i\}_{i=1}^{n}$ with $e_n=-\frac{\nabla f}{|\nabla f|}.$ Thus, from the second fundamental formula, for $1\leq a,b,c,d\leq n-1,$ one obtains that
\begin{equation*}
    h_{ab}=-\left\langle\nabla_{e_a}\nu,e_b\right\rangle=\frac{1}{|\nabla f|}\nabla_a\nabla_b f=0
\end{equation*}  and hence, $\partial M$ is totally geodesic. Thus, by the Gauss equation, i.e., 
\begin{eqnarray*}
R_{abcd}^{\partial M}=R_{abcd}-h_{ad}h_{bc}-h_{ac}h_{bd},
\end{eqnarray*} we then obtain

$$R^{\partial M}_{abcd}= R_{abcd}.$$ Moreover, we infer 

$$R^{\partial M}_{ac}= R_{ac}-R_{ancn}$$ and 

\begin{eqnarray}
\label{eqe1}
R^{\partial M} = R-2R_{nn}.
\end{eqnarray}

We recall that Proposition 2 in \cite{SPFS} asserts that the scalar curvature $R$ of a static perfect fluid space-time $(M^n,\,g,\,f,\rho)$ is constant if and only if $(\frac{1}{2}R+\rho)f$ is constant. Thus, since $M$ is a compact Riemannian manifold with (non-empty) boundary and $f_{|_{\partial M}}=0,$ one concludes that $R$ is constant if and only if $(\frac{1}{2}R+\rho)f\equiv 0.$ Therefore, by using that $f>0$ in the interior of $M,$ we infer that the scalar curvature is constant if and only if 
\begin{eqnarray}\label{eq-rho}
\rho=-\frac{R}{2}\;\mathrm{on}\;M.
\end{eqnarray} Notice furthermore that if $\rho=-\frac{R}{2}$ over $M,$ then the scalar curvature must be constant even in the empty boundary case.

As was mentioned in the introduction, the classical examples of SPFST are the standard hemisphere $\mathbb{S}^{n}_{+}$ (connected boundary) with standard metric and  the product $[0,\pi]\times\mathbb{S}^{n-1}$ (disconnected boundary) with metric $g=dt^2+(n-2)g_{_{\mathbb{S}^{n-1}}}.$ Now, we are going to describe all details concerning the new example of simply connected SPFST with boundary and constant scalar curvature stated in Example \ref{exA}; see also \cite[Example 2]{DGRPAMS} for more details.

\begin{example}[Example \ref{exA}]
\label{ex2}
	Let $M^n=\mathbb{S}^{p+1}_{+}\times\mathbb{S}^q$, $q>1$, with the doubly warped product metric 
	\begin{eqnarray*}
		g=dr^2+\sin^2(r)g_{\mathbb{S}^p}+\frac{q-1}{p+1}g_{\mathbb{S}^q},
	\end{eqnarray*}
	where $r(x,y)=r(x)$ is the height function of $\mathbb{S}^{p+1}_{+}.$ Moreover, we consider the potential function $f(r)=\cos(r)$ with $r\leq\frac{\pi}{2}.$ Thus, $(M^n,\,g)$ satisfies \eqref{eq-general} and \eqref{eq-laplacian}. In particular, since it has constant scalar curvature, then it is a static space.
	
	To check such an example, we first observe that 
	\begin{eqnarray*}
	\nabla f=-\sin(r)\nabla r.
\end{eqnarray*} From this, one deduces that 

\begin{eqnarray*}
\nabla^2 f=-\cos(r)dr^2-\cos(r)\sin^2(r)g_{\mathbb{S}^p}=-f(dr^2+\sin^2(r)g_{\mathbb{S}^p}).
\end{eqnarray*} Consequently, $$\nabla^2 f=-fg_{\mathbb{S}^{p+1}_{+}}.$$ In particular, one has
\begin{eqnarray*}
\Delta f=-(p+1)f.
\end{eqnarray*} Next, since $g=g_{\mathbb{S}^{p+1}_{+}}+\frac{q-1}{p+1}g_{\mathbb{S}^q}$ is a product metric, one obtains that
\begin{eqnarray}
\label{richjk9}
	Ric=p g_{\mathbb{S}^{p+1}_{+}}+(q-1) g_{\mathbb{S}^q};
\end{eqnarray}  for more details, see, e.g., Corollary 43 in \cite{O'Neil}. Thereby, consider a point $(x,y)\in \mathbb{S}^{p+1}_{+}\times\mathbb{S}^q$ and  an orthonormal basis $$\left\{E_{1},\ldots, E_{p+1},E_{p+2}=\sqrt{\frac{p+1}{q-1}}e_{1},\ldots, E_{p+q+1}=\sqrt{\frac{p+1}{q-1}}e_{q}\right\}$$ for $T_{(x,y)}(\mathbb{S}^{p+1}_{+}\times\mathbb{S}^q),$ where $\left\{E_{1},\ldots, E_{p+1}\right\}$ is an orthonormal frame for $T_{x}\mathbb{S}^{p+1}_{+}$ and $\left\{e_{1},\ldots, e_q \right\}$ is an orthonormal frame for $T_{y}\mathbb{S}^q.$ By tracing (\ref{richjk9}), one deduces that the scalar curvature is constant and given by 
\begin{eqnarray*}
R&=&\sum_{i=1}^{p+q+1}Ric(E_{i},E_{i})\nonumber\\&=&\sum_{i=1}^{p+1}p\,g_{\mathbb{S}^{p+1}_{+}}(E_{i},E_{i}) +\sum_{i=p+2}^{p+q+1}(q-1)g_{\mathbb{S}^q}(E_{i},E_{i})\nonumber\\&=&(p+q)(p+1)=(n-1)(p+1).
\end{eqnarray*} Of which, we arrive at
\begin{equation*}
-(\Delta f)g+\nabla^2 f-f Ric=(p+1)fg-fg_{\mathbb{S}^{p+1}_{+}}-f(pg_{\mathbb{S}^{p+1}_{+}}+(q-1)g_{\mathbb{S}^{q}})=0,
\end{equation*} which proves that  $\mathbb{S}^{p+1}_{+}\times\mathbb{S}^q$ is a static manifold with connected boundary $\partial M=\mathbb{S}^p\times\mathbb{S}^q$ and $f(r)=\cos(r)$ vanishes on the boundary. Furthermore, one easily verifies that $\mathbb{S}^{p+1}_{+}\times\mathbb{S}^q$ is simply connected.
	\end{example}
	
	\medskip

	\begin{remark} As previously mentioned, Example \ref{ex2} is a simply connected static space with positive scalar curvature and connected boundary. Therefore, it is a coun\-terexam\-ple to the Cosmic no-hair conjecture for arbitrary dimension $n\geq 4$.  	
	\end{remark}

Reasoning as in Example \ref{ex2}, we also obtain the following example of positive static triple. 

\begin{example}
We consider $M=[0,\pi]\times\mathbb{S}^p\times\mathbb{S}^q,$ $p$ and $q>1,$ endowed with the metric 
\begin{eqnarray*}
	g=dt^2+(p-1)g_{\mathbb{S}^p}+(q-1)g_{\mathbb{S}^q}.
\end{eqnarray*} Suppose that $f(t)=\sin(t).$ Hence, we deduce that $(M,\,g,\,f)$ is a positive static triple with disconnected boundary consisting of two copies of $\mathbb{S}^p\times\mathbb{S}^q,$ i.e., 
\begin{eqnarray*}
\partial M=(\{0\}\times\mathbb{S}^p\times\mathbb{S}^q)\cup(\{\pi\}\times\mathbb{S}^p\times\mathbb{S}^q).
\end{eqnarray*}

\end{example}

The following divergence formula established in \cite{SPFS} will play a key role in the proof of Theorem \ref{theo2}.

\begin{lemma}[\cite{SPFS}]\label{lemma1}
Let $(M^n,\,g)$ be a Riemannian manifold and $f$ is a smooth function sa\-tis\-fying \eqref{eqA}. Then, in the interior of $M,$ one has
\begin{eqnarray}\label{eq-RS}
{\rm div}\left[\frac{1}{f}\left(\nabla|\nabla f|^2-2\frac{\Delta f}{n}\nabla f\right)\right]=2f|\mathring{Ric}|^2+\frac{n-2}{n}\left\langle\nabla R,\nabla f\right\rangle.
\end{eqnarray}
\end{lemma}

Lemma \ref{lemma1} can be seen as a Robinson-Shen type identity \cite{Shen-FM} for manifolds satisfying Eq. (\ref{eqA}), which is a large class of spaces; see also \cite{Ambrozio, BM,BCM}. In the case of constant scalar curvature, for example, static space or $V$-static space, we deduce that the expression in the left hand side of \eqref{eq-RS} is necessarily non-negative.

Next, we recall the following integral formula that relates the norm of the traceless Ricci tensor and the scalar curvature of the boundary; for more details, see Lemma 3 of \cite{SPFS}. This formula will be used in the proof of the boundary estimates. 

\medskip

\begin{lemma}[\cite{SPFS}]
\label{lemma2}
Let $(M^n,\,g,f,\,\rho)$ be a compact oriented static perfect fluid space-time with boundary $\partial M.$ Then we have
\begin{eqnarray*}
\int_{\partial M}|\nabla f|R^{\partial M} dS_g=2\int_M f|\mathring{Ric}|^2 dV_g-\frac{n-2}{n}\int_M R\Delta f dV_g.
\end{eqnarray*}
\end{lemma}

\medskip

Proceeding, by using Lemmas \ref{lemma1} and \ref{lemma2} and assuming a suitable integral inequality, we shall deduce an area estimate for the boundary of a compact static perfect fluid space-time. The following proposition is related with Proposition 3 of \cite{SPFS}.

\begin{proposition}\label{propnew}
	Let $(M^n,\,g,\,f,\,\rho)$ be a compact oriented static perfect fluid space-time with connected boundary $\partial M.$ Suppose that 
	\begin{eqnarray}\label{eq-replace0}
		\int_{M} \langle\nabla R,\nabla f\rangle dV_g \geq 0.
	\end{eqnarray} Then, it holds 
	\begin{eqnarray}\label{eqr1}
		\int_{\partial M}R^{\partial M} dS_g\geq \frac{n-2}{n}R_{min}|\partial M|.
	\end{eqnarray} Furthermore, equality holds in (\ref{eqr1}) if and only if $(M^n,\,g)$ is an Einstein manifold. In particular, by assuming in addition that $R>0,$ equality holds in (\ref{eqr1}) if and only if $(M^n,\,g)$ is isometric to the round hemisphere $\mathbb{S}^{n}_{+}.$
\end{proposition}

\begin{proof}
From the Stokes' theorem, one sees that
	\begin{eqnarray*}
		\int_M (R-R_{min})\Delta f dV_g=-|\nabla f|_{|_{\partial M}}\int_{\partial M} (R-R_{min} )dS_g -\int_M  \langle\nabla R,\nabla f\rangle dV_g
	\end{eqnarray*}	 and by using the condition \eqref{eq-replace0}, one obtains that
\begin{eqnarray}\label{eqr2}	
\int_M (R-R_{min})\Delta f dV_g\leq 0.
\end{eqnarray} Next, by Lemma $\ref{lemma2}$ and again the Stokes' theorem, we infer
\begin{eqnarray}\nonumber
	|\nabla f|_{|_{\partial M}} \int_{\partial M} R^{\partial M} dS_g&=& 2\int_{M} f|\mathring{Ric}|^2 dV_g-\frac{n-2}{n}\int_{M} R\Delta f dV_g\\\nonumber
	&\geq& -\frac{(n-2)R_{min}}{n}\int_{M} \Delta f dV_g\\\label{eqr3}
	&=&|\nabla f|_{|_{\partial M}}\frac{n-2}{n}R_{min}|\partial M|,
\end{eqnarray} which proves \eqref{eqr1}. Moreover, equality occurs in the last expression if and only if $\mathring{Ric}\equiv 0,$ i.e., $M^n$ is Einstein. 
In particular, by assuming that $R>0,$ since $\partial M$ is totally geodesic, we may apply Proposition 1 of \cite{SPFS} to deduce that $(M^n,\,g)$ is isometric to, up to scaling, the standard hemisphere $\mathbb{S}^{n}_{+}.$ The converse assertion is straightforward.
So, the proof of the proposition is finished.
\end{proof}

\begin{remark}
By assuming the condition (\ref{eq-replace0}) and that the scalar curvature is positive, we may improve the conditions assumed in Theorem 1 and Corollary 3 of  \cite{SPFS}. To do so, it suffices to observe that, under these assumptions, the integral  $\int_{M} R\Delta f\,dV_g$ is negative.
\end{remark}

Now, we deal with the Reilly's formula. We recall the following generalized Reilly’s formula, obtained previously by Qiu and Xia \cite{Qiu-Xia}, that will
be very useful (see also \cite[Proposition 1]{DPR}).

\begin{proposition}[\cite{Qiu-Xia}]\label{prop-qui-xia}
Let $(M^n,\,g)$ be a compact Riemannian manifold with boundary $\partial M.$ Given two functions $f$ and $u$ on $M^n$ and a constant $\kappa$, we have 

\begin{eqnarray*}
&&\int_{M} f \left((\Delta u + \kappa nu)^{2}-|\nabla^{2}u+\kappa ug|^{2}\right)dV_g=(n-1)\kappa \int_{M}(\Delta f +n\kappa f)u^{2}\,dV_g \nonumber \\ &&+\int_{M}\left(\nabla ^{2}f-(\Delta f)g-2(n-1)\kappa f g+f Ric\right)(\nabla u, \nabla u)\,dV_g \nonumber \\
&&+\int_{\partial M}f \left[2\left(\frac{\partial u}{\partial \nu}\right)\Delta_{_{\partial M}}u+H\left(\frac{\partial u}{\partial \nu}\right)^{2}+ h(\nabla_{_{\partial M}}u, \nabla_{_{\partial M}} u)+2(n-1)\kappa \left(\frac{\partial u}{\partial \nu}\right)u\right]dS_g\nonumber\\
&&+ \int_{\partial M}\frac{\partial f}{\partial \nu}\left(|\nabla _{_{\partial M}}u|^{2}-(n-1)\kappa u^{2}\right) dS_g,
\end{eqnarray*} 
where $h$ and $H=tr_{g}h$ stand for the second fundamental form and mean curvature of $\partial M,$ respectively.
 \end{proposition}

 We point out that by considering $\kappa=0$ and $f=1$ in Proposition \ref{prop-qui-xia}, we recover the classical Reilly’s formula. We also refer the reader to \cite[Proposition 1]{DPR} for a detailed proof. Next, we are going to establish a key lemma that will be fundamental in the proof of Theorem \ref{th1B}.

\begin{lemma}
\label{lemplk1}
Let $\big(M^{n},\,g \big)$ be a compact manifold with boundary $\partial M$. We assume that $f$ is a smooth function on $M^n$ satisfying $$f\mathring{Ric}=\mathring{\nabla}^2 f\,\,\,\,\,\,\,\,\,\hbox{and}\,\,\,\,\,\,\,\,\,\,f_{\mid_{\partial M}}=0.$$ Then
\begin{eqnarray}
\label{plk1}
\int_{\partial M}\frac{\partial f}{\partial \nu}\left(|\nabla_{_{\partial M}}\eta|^2 -(n-1)\kappa \eta^2 \right) dS_g &=& \frac{1}{n}\int_M[(n-1)\Delta f+Rf]g(\nabla u,\nabla u)dV_g\nonumber\\
&&-(n-1)\kappa\int_M(\Delta f+n\kappa f)u^2dV_g\nonumber\\
& &-\int_{M}f|\nabla^2 u+\kappa u g|^2 dV_{g}\nonumber\\
&&-2\int_Mf[Ric-(n-1)\kappa g](\nabla u, \nabla u) dV_g,
\end{eqnarray}  
where $\eta$ is any function on $\partial M$ and $u$ is a solution of
\begin{eqnarray}\label{The initial value problemAS}
\left\{\begin{array}{rc}
\Delta u +n\kappa u = 0 &\mbox{ in }\,\,\, M,\\
u= \eta & \mbox{ on }\,\,\, \partial M.
\end{array}\right.
\end{eqnarray}

\end{lemma}
\begin{proof}
By Proposition \ref{prop-qui-xia}, Eq. (\ref{The initial value problemAS}) and the fact that $f=0$ on $\partial M,$ one sees that
\begin{eqnarray*}
\int_{\partial M}\frac{\partial f}{\partial \nu}\left(|\nabla_{_{\partial M}}\eta|^2 -(n-1)\kappa \eta^2 \right) dS_g&=&-\int_{M}f|\nabla^2 u+\kappa u g|^2 dV_{g}\\
&&-(n-1)\kappa\int_M(\Delta f+n\kappa f)u^2dV_g\\
&&-\int_{M}\left(\nabla ^{2}f-(\Delta f)g+f Ric\right)(\nabla u, \nabla u)\,dV_g\\
&&+2(n-1)\kappa\int_M f |\nabla u|^2\,dV_g.
\end{eqnarray*}
Now, using $f\mathring{Ric}=\mathring{\nabla^2}f,$ we get
\begin{eqnarray*}
\int_{\partial M}\frac{\partial f}{\partial \nu}\left(|\nabla_{_{\partial M}}\eta|^2 -(n-1)\kappa \eta^2 \right) dS_g&=&-\int_{M}f|\nabla^2 u+\kappa u g|^2 dV_{g}\\
&&-(n-1)\kappa\int_M(\Delta f+n\kappa f)u^2dV_g\\
&&-\int_{M}\left(2fRic-\frac{R}{n}fg-\frac{n-1}{n}\Delta fg\right)(\nabla u, \nabla u)\,dV_g\\
&&+2(n-1)\kappa\int_M f |\nabla u|^2\,dV_g\\
 &=&\frac{1}{n}\int_M[(n-1)\Delta f+Rf]g(\nabla u,\nabla u)dV_g\nonumber\\
&&-(n-1)\kappa\int_M(\Delta f+n\kappa f)u^2dV_g\nonumber\\
& &-\int_{M}f|\nabla^2 u+\kappa u g|^2 dV_{g}\nonumber\\
&&-2\int_Mf[Ric-(n-1)\kappa g](\nabla u, \nabla u) dV_g,
\end{eqnarray*} as we wanted to prove.

\end{proof}

To conclude this section, we shall establish a lemma that will be fundamental in the proof of the boundary estimate involving the Brown-York mass. To do so, similar to \cite{DGR,Yuan}, we set the conformal metric $\overline{g}=v^{\frac{4}{n-2}}g$ with 
\begin{equation}\label{conformal}
v=(1+\alpha f)^{-\frac{n-2}{2}}\,\,\,\,\,\,\,\;\mathrm{and}\,\,\,\,\,\,\,\;\alpha^{-1}=\max_M \left(f^2+\frac{n(n-1)}{R_g}|\nabla f|^2\right)^{\frac{1}{2}}. 
\end{equation} With aid of these notations, we have the following lemma.

\begin{lemma}\label{lmby}
Let $(M^n,\,g,\,f,\,\rho)$, $n\geq 3,$ be a compact static perfect fluid space-time with positive scalar curvature and satisfying the dominant energy condition. Then the scalar curvature $R_{\overline{g}}$ with respect to the conformal metric $\overline{g}$ is non-negative. Moreover, $R_{\overline{g}}=0$ if and only if $\Delta_g f=-\frac{R_{g}}{n-1}f$ and $f^2+\frac{n(n-1)}{R_{g}}|\nabla f|^2$ is constant on $M^n$. 
\end{lemma}

\begin{proof} Initially, by using the conformal metric defined above, a direct computation yields
\begin{eqnarray}\nonumber
    \Delta_g v&=&\;\Delta[(1+\alpha f)^{-\frac{n-2}{2}}]\\\nonumber
    &=&\;{\rm div}\left[-\frac{n-2}{2}(1+\alpha f)^{-\frac{n}{2}}\alpha\nabla f\right]\\\label{eqq1}
    &=&\; -\frac{n-2}{2}\alpha(1+\alpha f)^{-\frac{n}{2}}\Delta f+\frac{n(n-2)}{4}\alpha^2 (1+\alpha f)^{-\frac{n+2}{2}}|\nabla f|^2.
\end{eqnarray}

On the other hand, it follows from Eq. \eqref{eqB} and the dominant energy condition that
\begin{eqnarray}
\label{eqiguld}
    \Delta f&=&\;\left(\frac{n-2}{2(n-1)}R_g+\frac{n}{n-1}\rho\right)f\nonumber\\
    &\geq&\;\left(\frac{n-2}{2(n-1)}R_g-\frac{n}{2(n-1)}R_g\right)f\nonumber\\
    &=&\; -\frac{R_g}{n-1}f.
\end{eqnarray} This combined with \eqref{eqq1} gives
\begin{eqnarray*}
    \Delta_g v&\leq &\frac{n-2}{2(n-1)}\alpha(1+\alpha f)^{-\frac{n}{2}}R_g f\nonumber\\
    & &+\frac{n(n-2)}{4}\alpha^2 (1+\alpha f)^{-\frac{n+2}{2}}|\nabla f|^2.
\end{eqnarray*} Rearranging the terms, one sees that

\begin{eqnarray}\label{eqq2}
    \Delta_g v\leq \frac{n(n-2)}{4}\alpha(1+\alpha f)^{-\frac{n+2}{2}}\left[\frac{2R_g}{n(n-1)}(1+\alpha f)f+\alpha|\nabla f|^2\right].
\end{eqnarray}

In another direction, it is well known, by the formulae for conformal metric, that 

\begin{eqnarray*}
    R_{\overline{g}}=v^{-\frac{n+2}{n-2}}\left(R_g v-4\frac{n-1}{n-2}\Delta_g v\right),
\end{eqnarray*} where $R_{\overline{g}}$ is the scalar curvature with respect to the conformal metric $\overline{g}.$ From this, one obtains that

\begin{eqnarray*}
    \Delta_g v&=&\;\frac{n-2}{4(n-1)}R_g v-\frac{n-2}{4(n-1)}v^{\frac{n+2}{n-2}}R_{\overline{g}}\\
    &=&\;\frac{n-2}{4(n-1)}R_g(1+\alpha f)^{-\frac{n-2}{2}}-\frac{n-2}{4(n-1)}((1+\alpha f)^{-\frac{n-2}{2}})^{\frac{n+2}{n-2}}R_{\overline{g}}\\
    &=&\;\frac{n-2}{4(n-1)}R_g (1+\alpha f)^{-\frac{n-2}{2}}-\frac{n-2}{4(n-1)}(1+\alpha f)^{-\frac{n+2}{2}}R_{\overline{g}}\\
    &=&\;\frac{n-2}{4(n-1)}(1+\alpha f)^{-\frac{n+2}{2}}\left[(1+\alpha f)^2R_g-R_{\overline{g}}\right].
\end{eqnarray*} Substituting this into \eqref{eqq2} and rearranging the terms, one deduces that

\begin{eqnarray*}
    R_g(1+\alpha f)^2-R_{\overline{g}}\leq n(n-1)\alpha \left[\frac{2 R_g}{n(n-1)}(1+\alpha f)f+\alpha|\nabla f|^2\right].
\end{eqnarray*} Of which, we obtain

\begin{eqnarray}\nonumber
    R_{\overline{g}}&\geq&\; R_g(1+\alpha f)^2-n(n-1)\alpha \left[\frac{2 R_g}{n(n-1)}(1+\alpha f)f+\alpha|\nabla f|^2\right]\\\nonumber
    &=&\;R_g\left[(1+\alpha f)^2-2\alpha(1+\alpha f)f-\frac{n(n-1)}{R_g}\alpha^2|\nabla f|^2\right]\\\label{eqq3}
    &=&\;R_g\left[1-\alpha^2\left(f^2+\frac{n(n-1)}{R_g}|\nabla f|^2\right)\right].
\end{eqnarray} Therefore, by using \eqref{eqq3} and the value chosen for $\alpha$ in \eqref{conformal}, one concludes that $R_{\overline{g}}\geq 0,$ as asserted.  

Finally, observe that $R_{\overline{g}}=0$ if and only if equality holds in \eqref{eqq3} and (\ref{eqiguld}). Consequently, $\Delta f=-\frac{R_g}{n-1}f$ and since $R_{g}>0,$ one sees that $f^2+\frac{n(n-1)}{R_g}|\nabla f|^2=\alpha^{-2}$ is constant over $M$, which in particular implies $\rho=-\frac{R_g}{2}$ and hence, $R_g$ is constant. So, the proof is finished.
\end{proof}

\section{Proof of the Main Results}

In this section, we shall present the proofs of Theorems \ref{th1B}, \ref{theo2} and \ref{theo2.1} and Corollaries \ref{cor1} and \ref{coro-BY}.

\subsection{Proof of Theorem \ref{th1B}}

\begin{proof}

To begin with, we have from  (\ref{eq-laplacian}) that 
\begin{equation}
\label{eqLamK}
\Delta f=-\Lambda f,
\end{equation} where $\Lambda=-\frac{(n-2)}{2(n-1)}R-\frac{n}{n-1}\rho.$ In particular, it follows from Remark \ref{remarkLambda} that $\Lambda>0.$ Now, choosing $u=f$ and $\kappa=\frac{\Lambda}{n},$ one obtains that
\begin{eqnarray*}
\left\{\begin{array}{rc}
\Delta u +n\kappa u = 0 &\mbox{ in } M,\\
u= 0 & \mbox{ on } \partial M.
\end{array}\right.
\end{eqnarray*} Hence, by using Lemma \ref{lemplk1}, one sees that
		\begin{eqnarray}\label{th1-est03}
		\int_Mf\Big|\nabla^2f+\frac{\Lambda}{n}fg\Big|^2dV_g&-&\frac{1}{n}\int_M Rf|\nabla f|^2dV_g
		-\Big(\frac{n-1}{n}\Big)\int_M(\Delta f)|\nabla f|^2dV_g\nonumber\\
		   &+&2\int_Mf\Big[Ric-\Big(\frac{n-1}{n}\Big)\Lambda g\Big](\nabla f,\nabla f)dV_g=0.
		\end{eqnarray}

	In another direction, by the classical Bochner's formula $$\frac{1}{2}\Delta |\nabla f|^2 = Ric(\nabla f, \nabla f)+|\nabla^2 f|^2 + \langle \nabla f,\,\nabla \Delta f\rangle,$$ one deduces that

		\begin{equation*}
		    2fRic(\nabla f,\nabla f)=f\Delta |\nabla f|^2-2f|\nabla^2f|^2+2\Lambda f|\nabla f|^2.
		\end{equation*} Upon integrating, one sees that
		\begin{eqnarray}\label{th1-est04}
		2\int_Mf\Big[Ric-\Big(\frac{n-1}{n}\Big)\Lambda g\Big](\nabla f,\nabla f)dV_g
		&=&\int_Mf\Delta|\nabla f|^2dV_g-2\int_Mf|\nabla^2f|^2dV_g\nonumber\\
		&&+\frac{2\Lambda}{n}\int_Mf|\nabla f|^2dV_g.
		\end{eqnarray} Also, by Green's identity and the fact that $f=0$ on $\partial M,$ it is easy to check that
	\begin{equation}\label{th1-est05}
	    \int_Mf\Delta |\nabla f|^2dV_g=\int_M|\nabla f|^2\Delta fdV_g+|\nabla f|^3_{|_{\partial M}}|\partial M|.
	\end{equation} Substituting (\ref{th1-est05}) into (\ref{th1-est04}) yields
\begin{eqnarray*}
2\int_M f\Big[Ric-\Big(\frac{n-1}{n}\Big)\Lambda g\Big](\nabla f, \nabla f)\,dV_g &=&\int_M|\nabla f|^2\Delta fdV_g+|\nabla f|^3_{|_{\partial M}}|\partial M|\nonumber
\\&&-2\int_Mf|\mathring{\nabla^2}f|^2\,dV_g-\frac{2}{n}\int_Mf(\Delta f)^2\,dV_g
\nonumber\\
 &&+\frac{2\Lambda}{n}\int_Mf|\nabla f|^2\,dV_g,
\end{eqnarray*} where we have used that $|\nabla^2f|^2=|\mathring{\nabla^2}f|^2+\frac{(\Delta f)^2}{n}.$ Using (\ref{th1-est03}), we then obtain

\begin{eqnarray}\label{th1-est07}
\int_Mf|\mathring{\nabla^2}f|^2dV_g&=&-\frac{1}{n}\int_M Rf|\nabla f|^2dV_g+\frac{1}{n}\int_M(\Delta f)|\nabla f|^2dV_g\nonumber\\
& &+|\nabla f|^3_{|_{\partial M}}|\partial M|
-\frac{2}{n}\int_Mf(\Delta f)^2dV_g+\frac{2\Lambda}{n}\int_M f|\nabla f|^2dV_g.
\end{eqnarray}

Proceeding, integrating by parts using (\ref{eqLamK}) and $f=0$ on $\partial M,$ one deduces that
\begin{equation}\label{th1-est07-1}
    \int_Mf(\Delta f)^2\,dV_g=-\Lambda\int_Mf^2\Delta f\,dV_g=2\Lambda\int_Mf|\nabla f|^2\,dV_g.
\end{equation} Plugging this together with (\ref{eqLamK}) into (\ref{th1-est07}) gives

\begin{eqnarray}\label{th1-est08}
\int_Mf|\mathring{\nabla^2}f|^2dV_g&=&-\frac{1}{n}\int_MfR|\nabla f|^2dV_g+|\nabla f|^3_{|_{\partial M}}|\partial M|\nonumber\\
& &-\frac{3\Lambda}{n}\int_Mf|\nabla f|^2dV_g.
\end{eqnarray} Since $-\Delta f>0$ (in the interior of $M$), by H\"older's inequality, one sees that

\begin{eqnarray*}
   \left(\int_M f \Delta f\,dV_g\right)^{2}&\leq & \int_{M}f^2 \left(-\Delta f\right)\,dV_{g}\,\int_{M}\left(-\Delta f\right)\,dV_{g}\nonumber\\&=& \int_{M}f^2 \left(-\Delta f\right)\,dV_{g}\,|\nabla f|_{|_{\partial M}}|\partial M|,
\end{eqnarray*} where we used the Stokes' formula. With aid of (\ref{th1-est07-1}), we infer

\begin{eqnarray}
\label{th1-est09}
   |\nabla f|_{|_{\partial M}}|\partial M|\, \int_Mf|\nabla f|^2dV_g&=&\frac{1}{2}|\nabla f|_{|_{\partial M}}|\partial M|\,\int_Mf^2(-\Delta f)dV_g\nonumber\\&\geq & \frac{1}{2}\Big(\int_M f\Delta f dV_g\Big)^2.
\end{eqnarray}

Again, by H\"older's inequality, we get
\begin{eqnarray*}
    \Big(\int_M \Delta fdV_g\Big)^2&\leq & \left(\int_{M}\Lambda dV_{g}\right)\,\left(\int_{M}f\big(-\Delta f\big)dV_{g}\right)\nonumber\\&=&\Lambda Vol(M)\left(\int_{M}f\big(-\Delta f\big)dV_{g}\right),
\end{eqnarray*} so that

\begin{equation}
\label{jk4tbn8}
\Lambda^2 Vol(M)^2 \left(\int_{M}f\big(\Delta f\big)dV_{g}\right)^2 \geq \Big(\int_M \Delta fdV_g\Big)^4= |\nabla f|^4_{|_{\partial M}}|\partial M|^4.
\end{equation}

Combining (\ref{jk4tbn8}) and (\ref{th1-est09}), one obtains that

\begin{equation*}
 2\Lambda^2Vol(M)^2   \int_Mf|\nabla f|^2dV_g\geq |\nabla f|^3_{|_{\partial M}}|\partial M|^3.
\end{equation*} Substituting this into (\ref{th1-est08}), one sees that

\begin{eqnarray*}
 0&\leq&\int_M f|\mathring{\nabla^2}f|^2dV_g\\
  &=&-\frac{1}{n}\int_M Rf|\nabla f|^2dV_g+|\nabla f|^3_{|_{\partial M}}|\partial M|-\frac{3\Lambda}{n}\int_M f|\nabla f|^2dV_g\\
  &\leq&-\frac{R_{min}+3\Lambda}{n}\int_M f|\nabla f|^2dV_g+|\nabla f|^3_{|_{\partial M}}|\partial M|\\
  &\leq&-\frac{(R_{min}+3\Lambda)|\nabla f|^3_{|_{\partial M}}|\partial M|^3}{2n\Lambda^2Vol(M)^2}+|\nabla f|^3_{|_{\partial M}}|\partial M|.
\end{eqnarray*} From this, it follows that
\begin{eqnarray*}
Vol(M)\geq\frac{1}{\Lambda}\sqrt{\frac{R_{min}+3\Lambda}{2n}}\,|\partial M|,
\end{eqnarray*} as asserted. 

To conclude, if the equality holds, then

\begin{equation*}
    \nabla^2f=-\frac{\Lambda}{n}fg.
\end{equation*} Moreover, since $ f=0$ on the boundary $\partial M$ ($f$ is constant on $\partial M$), we may use Theorem B  (II) of \cite{Reilly} to deduce that $(M^n,\,g)$ has constant sectional curvature $\frac{\Lambda}{n}$ and this forces $(M^n,\,g)$ to be an Einstein manifold. Thus, it suffices to invoke Proposition 1 of  \cite{SPFS} together with the fact that the boundary is totally geodesic in order to conclude that $(M^n,\,g)$ is isometric to the round hemisphere $\mathbb{S}^{n}_{+}.$ This finishes the proof of the theorem.
\end{proof}

\subsection{Proof of Corollary \ref{cor1}}

	\begin{proof}
	By Proposition 2 in \cite{SPFS}, the condition that $\big(M^{n},\,g,\,f,\,\rho\big)$ has positive constant scalar curvature $R$ implies that $\rho=-\frac{R}{2}$ and $\Lambda=\frac{R}{n-1}$ is constant. Therefore, Corollary \ref{cor1} follows from Theorem \ref{th1B}.
	\end{proof}

\subsection{Proof of Theorem \ref{theo2}}

\begin{proof}

 In the first part of the proof, we shall follow the arguments in \cite{DGR,Fang-Yuan,Yuan}. First of all, by (\ref{conformal}), we have 
 \begin{eqnarray}\label{conformal2}
\,\,\,\,\,\,\,\,\,\,v=(1+\alpha f)^{-\frac{n-2}{2}}\;\,\,\,\,\,\,\mathrm{and}\,\,\,\,\,\,\,\;\alpha^{-1}=\max_M \left(f^2+\frac{n(n-1)}{R_g}|\nabla f|^2\right)^{\frac{1}{2}}.  
\end{eqnarray} We now claim that the mean curvature $H_{\overline{g}}^{i}$ of $\partial M_i$ with respect to the conformal metric $\overline{g}=v^{\frac{4}{n-2}}g$ is strictly positive. Indeed, since $f|_{\partial M}=0,$ it follows from \eqref{conformal2} that $v\mid_{\partial M}=1.$ Of which, $\overline{g}=g$ over the boundary and $(\partial M_i,\overline{g})$ is isometric to $(\partial M_i,\,g)$, which by assumption can be isometrically embedded in $\mathbb{R}^n$ as a convex hypersurface with mean curvature $H_{0}^{i},$ induced by the Euclidian metric. Besides, taking into account that mean curvature of $\partial M_i$ with respect to the conformal metric $\overline{g}$ is given by
\begin{eqnarray}\label{eqq5}
    H_{\overline{g}}^{i}=v^{-\frac{2}{n-2}}\left(H_{g}^{i}+2\frac{n-1}{n-2}\partial_\nu(\log(v))\right),
\end{eqnarray} one obtains that 

\begin{eqnarray}
\label{eqq6}
    H_{\overline{g}}^{i}&=& \frac{2(n-1)}{n-2}\big\langle \nabla(1+\alpha f)^{-\frac{n-2}{2}},\,\nu\big\rangle\nonumber\\
    &=& (n-1)\alpha|\nabla f|_{|_{\partial M_i}},
\end{eqnarray} where we have used that $H_{g}^{i}=0$ and $\nu=-\frac{\nabla f}{|\nabla f|}.$ This proves that $H_{\overline{g}}^{i}>0,$ as claimed.

Proceeding, from \eqref{eqq6} we get

\begin{eqnarray}
\label{eqq7}
    \mathfrak{m}_{BY}(\partial M_i,\overline{g})&=& \int_{\partial M_i}(H_{0}^{i}-H_{\overline{g}}^{i}) dS_g\nonumber\\
    &=& \mathfrak{m}_{BY}(\partial M_i,g)-(n-1)\alpha|\nabla f|_{|_{\partial M_i}}|\partial M_i|.
\end{eqnarray} Hence, by Lemma \ref{lmby} we deduce that $R_{\overline{g}}\geq 0.$ Moreover, since $H_{\overline{g}}^{i}>0,$ we may invoke the Riemannian Positive Mass Theorem for the Brown-York mass \cite[Theorem 1.1]{Shi-Tam} (see also \cite[Theorem 1]{Yuan}) to conclude that $\mathfrak{m}_{BY}(\partial M_i,\overline{g})\geq 0.$ Consequently, 
\begin{eqnarray}
\label{eqq8}
    |\partial M_i|&\leq &\frac{1}{(n-1)\alpha|\nabla f|_{|_{\partial M_i}}}\mathfrak{m}_{BY}(\partial M_i,g)\nonumber\\
    &=&\frac{1}{(n-1)\alpha|\nabla f|_{|_{\partial M_i}}}\int_{\partial M_i}H_{0}^{i} dS_g,
\end{eqnarray} which proves \eqref{eqq4}.

Next, if the equality holds in \eqref{eqq8} for some component $\partial M_{i_0}$, then we necessarily have
\begin{eqnarray*}
    \mathfrak{m}_{BY}(\partial M_{i_0},\overline{g})=0.
\end{eqnarray*} So, by invoking the equality case in the Riemannian Positive Mass Theorem for the Brown-York mass, one deduces  that the conformal metric $\overline{g}$ is flat and consequently, $(M^n,\,\overline{g})$ is isometric to a bounded domain in $\mathbb{R}^n$. Besides, $R_{\overline{g}}=0$ and then it follows from Lemma \ref{lmby} that $\Delta_g f=-\frac{R_g}{n-1}f$ and $\frac{n(n-1)}{R_g}|\nabla f|^2+f^2$ is constant. In particular, $\rho=-\frac{R_g}{2}.$ But, in this case, $R_g$ must be constant over $M.$ Therefore, one obtains that
\begin{eqnarray*}
    0=\nabla\left[\frac{n(n-1)}{R_g}|\nabla f|^2+f^2\right]=\frac{2n(n-1)}{R_g}\nabla^2 f(\nabla f)+2f\nabla f
\end{eqnarray*}
so that
\begin{eqnarray*}
\nabla|\nabla f|^2-\frac{2\Delta_g f}{n}\nabla f=0.
\end{eqnarray*} Now, it suffices to use Lemma \ref{lemma1} to conclude that $|\mathring{Ric}|^2=0$ in $M.$ Hence, we may use Proposition 1 of \cite{SPFS} and the fact that $\partial M$ is totally geodesic in order to conclude that $(M^n,\,g)$ is isometric to $\mathbb{S}^{n}_{+}$.

On the other hand, if $(M^n,\,g)$ is isometric to $\mathbb{S}^{n}_{+},$ one deduces that $\partial M=\mathbb{S}^{n-1}.$ Thereby, the Brown-York mass of $\mathbb{S}^{n-1}$ is given by
\begin{eqnarray*}
	\mathfrak{m}_{BY}(\mathbb{S}^{n-1})=\int_{\mathbb{S}^{n-1}}(n-1) dS_{g_{\mathbb{S}^{n-1}}}=(n-1)\omega_{n-1},
\end{eqnarray*}  where $\omega_{n-1}$ is the volume of the round $(n-1)$-dimensional sphere. At the same time, since $\mathbb{S}_+^n$ has constant scalar curvature $R_g=n(n-1),$ it follows from (\ref{eqB}) that $\Delta_gf=-\frac{R_g}{n-1}f=-nf$ and consequently, $f^2+\frac{n(n-1)}{R_g}|\nabla f|^2$ is constant on $M.$ Indeed, a direct computation using (\ref{eqA}) yields
\begin{eqnarray*}
\nabla\left(f^2+\frac{n(n-1)}{R_g}|\nabla f|^2\right)&=&2f\nabla f+2\nabla^2f(\nabla f)\\
 &=&-\frac{2\Delta f}{n}\nabla f+2\nabla^2f(\nabla f)\\
 &=&2\mathring{\nabla^2}f(\nabla f)=2f\mathring{Ric}(\nabla f)=0.
\end{eqnarray*} Hence, one obtains that 
\begin{eqnarray*}
	\alpha^{-2}=\left.\left(f^2+\frac{n(n-1)}{R_g}|\nabla f|^2\right)\right|_{\partial M}=|\nabla f|^{2}_{|{\partial M}},
\end{eqnarray*} so that 

$$\alpha=\frac{1}{|\nabla f|_{|{\partial M}}}.$$ Of which, we arrive at 
\begin{eqnarray*}
	\frac{\mathfrak{m}_{BY}(\partial M,g)}{\alpha (n-1)|\nabla f|_{|_{\partial M}}}=\omega_{n-1}=|\partial M|,
\end{eqnarray*} which is the equality in \eqref{eqq8}. So, the proof is completed.
\end{proof}

\subsection{Proof of Corollary \ref{coro-BY}}
\begin{proof}
Initially, we assume that $(\partial M_i,\overline{g})$ and $(\partial M_i,g)$ are isometric. In this case, one has $R^{\partial M_i}_{\overline{g}}=R^{\partial M_i}_{g}$. By using the Gauss' equation for $\partial M_i$ as an embedded hypersurface of $\mathbb{R}^n,$ one obtains that
\begin{eqnarray}\label{by1}
	R^{\partial M_i}_{g}=(H_{0}^{i})^2-|h_i|^2=\frac{n-2}{n-1}(H_{0}^{i})^2-|\mathring{h}_i|^2,
\end{eqnarray}
where $h_i$ and $\mathring{h}_i$ stand for the second fundamental form and traceless second fundamental form of $\partial M_i$, respectively. Now, we use \eqref{eqq8} and the H\"older's inequality in order to infer
\begin{eqnarray}\nonumber
|\partial M_i|&\leq & \frac{1}{(n-1)^2\alpha^2|\nabla f|_{|_{\partial M_i}}^{2}}\int_{\partial M_i}(H_{0}^{i})^2 dS_g\\\label{by2}
 &=& \frac{1}{(n-1)(n-2)\alpha^2|\nabla f|_{|_{\partial M_i}}^{2}}\int_{\partial M_i}(R^{\partial M_i}_{g}+|\mathring{h}_i|^2) dS_g,
\end{eqnarray} where we have used (\ref{by1}).

Clearly, if equality holds in \eqref{by2}, then \eqref{eqq8} also becomes an equality. Hence, one concludes that $(M^n,\,g)$ is isometric to the round hemisphere $\Bbb{S}^{n}_{+}.$

On the other hand, if $M$ is isometric to $\mathbb{S}^{n}_{+}$ with standard metric, then $R^{\partial M}_{g}=(n-2)(n-1)$ and $\mathring{h}_i=0.$ Furthermore, one has $$\big(f^2+\frac{n(n-1)}{R_g}|\nabla f|^2\big)\mid_{_{\partial M}}=|\nabla f|^{2}_{|_{\partial M}}$$ and consequently, we obtain
\begin{eqnarray*}
	\frac{1}{(n-1)(n-2)\alpha^2|\nabla f|_{|_{\partial M_i}}^{2}}\int_{\partial M_i}(R^{\partial M_i}_{g}+|\mathring{h}_i|^2) dS_g=\omega_{n-1}=|\partial M|,
\end{eqnarray*} which gives the equality in \eqref{by2}. This concludes the proof of the corollary.
\end{proof}

\subsection{Proof of Theorem \ref{theo2.1}}
\begin{proof}
To begin with, since $\partial M$ is totally geodesic, a direct computation using \eqref{eq-egvle} and the Stokes' formula yields 
\begin{eqnarray*}
\lambda_1\int_{\partial M}\varphi^2 dS_g\leq-\int_{\partial M}\varphi J(\varphi) dS_g=\int_{\partial M}|\nabla^{\partial M}\varphi|^2 dS_g-\int_{\partial M}R_{nn} \varphi^2 dS_g,
\end{eqnarray*} for all $\varphi\in C^{\infty}(\partial M).$ In particular, choosing  $\varphi\equiv 1$ and using \eqref{eqe1}, one deduces that

\begin{eqnarray}\label{eq-last}
\lambda_1|\partial M|\leq-\int_{\partial M}R_{nn} dS_g=-\int_{\partial M}\frac{R-R^{\partial M}}{2} dS_g.
\end{eqnarray}

From the assumption that the boundary is Einstein and $Ric^{\partial M}>0,$ one sees that $$Ric^{\partial M}= (n-2)\varepsilon,$$ where $\varepsilon=\frac{R^{\partial M}}{(n-1)(n-2)}.$ Therefore, we may invoke the Bonnet-Myers' theorem to infer that $$diam_{g_{\partial M}}(\partial M)\leq\frac{\pi}{\sqrt{\varepsilon}}.$$ At the same time, by Bishop-Gromov's theorem we have
\begin{eqnarray}\label{Bishop}
    Vol(B_{\frac{\pi}{\sqrt\varepsilon}}^{\partial M})\leq Vol(\mathbb{S}^{n-1}_{g_{\varepsilon}}),
\end{eqnarray}
where $g_{\varepsilon}=\frac{1}{\varepsilon}g_{\mathbb{S}^{n-1}}$. Next, since $|\partial M|\leq Vol(B_{\frac{\pi}{\sqrt{\varepsilon}}}^{\partial M})(p)$, for any point $p\in\partial M$ and $Vol(\mathbb{S}^{n-1}_{g_\varepsilon})=\varepsilon^{-\frac{n-1}{2}}\omega_{n-1},$ we then obtain
\begin{eqnarray}\label{Bishop2}
    |\partial M|^{\frac{2}{n-1}}\leq\frac{1}{\varepsilon}(\omega_{n-1})^{\frac{2}{n-1}},
\end{eqnarray} where $\omega_{n-1}$ denotes the volume of the standard unit sphere $\mathbb{S}^{n-1}.$ Substituting the expression of $\varepsilon$ into (\ref{Bishop2}), we arrive at
\begin{eqnarray*}
(\omega_{n-1})^{\frac{2}{n-1}}\geq \frac{R^{\partial M}}{(n-1)(n-2)}|\partial M|^{\frac{2}{n-1}},
\end{eqnarray*}
which can be rewritten as
\begin{eqnarray*}
R^{\partial M}\leq (n-1)(n-2) (\omega_{n-1})^{\frac{2}{n-1}}|\partial M|^{-\frac{2}{n-1}}.
\end{eqnarray*} Thereby, it suffices to use \eqref{eq-last} to infer

\begin{equation}\label{eq-eigen}
\lambda_1\leq- \frac{1}{2}\left(R_{*}-(n-1)(n-2)(\omega_{n-1})^{\frac{2}{n-1}}|\partial M|^{-\frac{2}{n-1}}\right),
\end{equation} where $R_{*}=\min\{R(p);\,p\in\partial M\}.$ This proves the asserted inequality.

To conclude, by assuming that equality occurs in \eqref{eq-eigen}, it is easy to check that \eqref{Bishop} and \eqref{Bishop2} are also equalities. Thus, it follows from the equality case for the Bishop-Gromov's theorem that the boundary is isometric to a round sphere $\mathbb{S}^{n-1}.$ So, the proof is finished.  
\end{proof}

\end{document}